\documentclass[12pt]{article}
\usepackage{amsmath}
\usepackage{amssymb}
\usepackage{amsthm}
\usepackage{extarrows}

\newtheorem{theorem}{Theorem}

\newtheorem{lemma}{Lemma}
\newtheorem{definition}{Definition}

\newtheorem{proposition}{Proposition}

\newcommand{\Mk}{\mathcal{M}_k}
\newcommand{\Mkk}{\mathcal{M}_{k-1}}

\newcommand{\Rn}{\mathbb{R}^{n}}
\newcommand{\RN}{\mathbb{R}^{n+1}}
\newcommand{\Cln}{\mathcal{C}l_{n}}
\newcommand{\ClN}{\mathcal{C}l_{n}}

\newcommand{\Sn}{\mathbb{S}^{n-1}}
\newcommand{\SN}{\mathbb{S}^{n}}

\newcommand{\wzwo} {\stackrel{\circ}{{W}^1_2}}
\newcommand{\wzwop} {\stackrel{\circ}{{W}^k_p}}

\begin{document}

\title{Spherical $\Pi$-type Operators in Clifford Analysis and Applications }
\author{Wanqing Cheng, John Ryan and Uwe K\"{a}hler}
\date{}
\maketitle
\begin{abstract}
The $\Pi$-operator (Ahlfors-Beurling transform) plays an important role in solving the Beltrami equation. In this paper we define two $\Pi$-operators on the n-sphere. The first spherical $\Pi$-operator is shown to be an $L^2$ isometry up to isomorphism. To improve this, with the help of the spectrum of the spherical Dirac operator, the second spherical $\Pi$ operator is constructed as an isometric $L^2$ operator over the sphere. Some analogous properties for both $\Pi$-operators are also developed. We also study the applications of both spherical $\Pi$-operators to the solution of the spherical Beltrami equations.
\end{abstract}

\section{Introduction}
The $\Pi$-operator is one of the tools used to study smoothness of functions over Sobolev spaces and to solve some first order partial differential equations such as the Beltrami equation which describes quasi-conformal mappings. In one dimensional complex analysis, the Beltrami equation is the partial differential equation:
$$\displaystyle \frac{\partial w}{\partial \overline{z}}=\mu \displaystyle \frac{\partial w}{\partial z}$$
where $\mu=\mu(z)$ is a given complex function, and $z=x+iy\in\mathbb{C}$, $\partial_{z}=\displaystyle\frac{\partial}{\partial x}-i\displaystyle\frac{\partial}{\partial y}$, $\partial_{\bar{z}}=\displaystyle\frac{\partial}{\partial x}+i\displaystyle\frac{\partial}{\partial y}$. It can be transformed to a fixed-point equation
$$h=q(z)(I+\Pi_\Omega h)$$
where
$$\Pi_\Omega h(z)=-\displaystyle\frac{1}{\pi i}\displaystyle\int_\Omega \frac{h(\xi)}{(\xi-z)^2}d\xi_1 d\xi_2$$
is the complex $\Pi$-operator. This singular integral operator acts as an isometry from $L^2(\mathbb{C})$ to $L^2(\mathbb{C})$ with the $L_p$-norm being a long standing conjecture by Ivaniec.
\par
With the help of Clifford algebras, the classical Beltrami equation and $\Pi$-operator with some well known results can be generalized to higher dimensions. Abundant results in Euclidean space have been found( see \cite{Blaya,Kahler,GK}). In order to generate results in Euclidean space to the unit sphere, we define two $\Pi$-operators related to the conformally invariant spherical Dirac operator. The idea to consider the $n$-sphere is not only motivated by being the classic example of a manifold and being invariant under the conformal group, but also by the fact that in the case of $n=3$ due to the recently proved Poincar\'e conjecture  there is a wide class of manifolds which are homeomorphic to the 3-sphere. This makes our results much more general and valid for any simply connected closed 3-manifold. In particular, results on local and global homeomorphic solutions of the sperical Beltrami equation carry over to such manifolds. \\
\par
This paper is organized as follows: In section 2, we briefly introduce Clifford algebras, Clifford analysis, the Euclidean Dirac operator, and some well known integral formulas. In Section 3, we review the construction and some properties for the $\Pi$-operator in Euclidean space. In section 4, we construct the $\Pi$-operator in a generalized spherical space and solve the Beltrami equation with a singular integral operator $\Pi_{s,0}$. In the last section, we will investigate the spectra of several spherical Dirac type operators and the spherical Laplacian, and construct the isometric spherical $\Pi$-operator $\Pi_{s,1}$.\\
\\
\textbf{Dedication}: This paper is dedicated to Franciscus Sommen on the occasion of his 60th birthday.
\section{Preliminaries}
\par
Let $e_1,\cdots,e_n$ be an orthonormal basis of $\RN$. The Clifford algebra $\ClN$ is the algebra over $\Rn$ generated by the relation
$$x^2=-||x||^2e_0$$
where $e_0$ is the identity of $\ClN$. These algebras were introduced by Clifford in 1878 in \cite{Clifford}. Each element of the algebra $\ClN$ can be represented in the form
\begin{eqnarray*}
x=\sum_{A\subset\{1,\cdots,n\}}x_Ae_A
\end{eqnarray*}
where $x_A$ are real numbers. The norm of a Clifford number $x$ is defined as
 $$\|x\|^2=\sum_{A\subset\{1,\cdots,n\}}x_A^2.$$
 If the set $A$ contains $k$ elements, then we call $e_A$ a \emph{k-vector}. Likewise, we call each linear combination of $k$-vectors a $k$-vector. The vector space of all $k$-vectors is denoted by $\Lambda^k\Rn$. Obviously, $\ClN$ is the direct sum of all $\Lambda^k\Rn$ for $k\leq n$. The following anti-involutions are well known:
\begin{itemize}
\item Reversion:\\
\begin{eqnarray*}
\tilde{a}=\sum_{A} (-1)^{|A|(|A|-1)/2}a_Ae_A,
\end{eqnarray*}
where $|A|$ is the cardinality of $A$. In particular, $\widetilde{e_{j_1}\cdots e_{j_r}}=e_{j_r}\cdots e_{j_1}$. Also $\widetilde{ab}=\tilde{b}\tilde{a}$ for $a,\ b\in\mathcal{C}l_n$.
\item Clifford conjugation:\\
\begin{eqnarray*}
a^{\dagger}=\sum_{A} (-1)^{|A|(|A|+1)/2}a_Ae_A,
\end{eqnarray*}
satisfying ${e_{j_1}\cdots e_{j_r}}^{\dagger}=(-1)^re_{j_r}\cdots e_{j_1}$ and ${(ab)}^{\dagger}={b}^{\dagger}{a}^{\dagger}$ for $a,\ b\in\mathcal{C}l_n$.
\item Clifford involution:\\
\begin{eqnarray*}
\bar{a}=\tilde{a}^{\dagger}=\widetilde{a^{\dagger}}.
\end{eqnarray*}
\end{itemize}
In the following we identify the Euclidean space $\RN$ with the direct sum $\Lambda^0\Rn\oplus\Lambda^1\Rn$. For all that follows let $\Omega\subset \RN$ be a domain with a sufficiently smooth boundary $\Gamma=\partial \Omega$. Then functions $f$ defined in $\Omega$ with values in $\ClN$ are considered. These functions may be written as
\begin{eqnarray*}
f(x)=\sum_{A\subseteq\{ e_1,e_2,...e_n \}}e_Af_A(x),\ \ (x\in \Omega).
\end{eqnarray*}
Properties such as continuity, differentiability, integrability, and so on, which are ascribed to $f$ have to be possessed by all components $f_A(x),\ (A\subseteq\{ e_1,e_2,...e_n \})$. The spaces $C^k(\Omega,{\mathcal{C}l_n}), L_p(\Omega, {\mathcal{C}l_n}) $ are defined as right Banach modules with the corresponding traditional norms. The space $L_2(\Omega,{\mathcal{C}l_n})$ is a right Hilbert module equipped with a ${\mathcal{C}l_n}$-valued sesquilinear form
$$
(u,v)=\int_\Omega \overline{u(\eta)} v(\eta)\, d\Omega_\eta. 
$$ 
Furthermore, $W_p^k(\Omega,{\mathcal{C}l_n}), k\in \mathbb{N}\cup\{0\},1\leq p<\infty$ denotes the Sobolev spaces as the right module of all functionals whose derivatives belong to $L_p(\Omega,{\mathcal{C}l_n})$, with norm
  $$
  \|f\|_{W_p^k(\Omega,{\mathcal{C}l_n})}:=\left(\sum_{A}\sum_{\|\alpha\|\leq k}\|D^\alpha_w f_A\|_{L_p(\Omega,{\mathcal{C}l_n})}^p\right)^{1/p}.
  $$ 
The closure of the space of test functions $C^\infty_0(\Omega, {\mathcal{C}l_n})$ in the $W_p^k$-norm will be denoted by $\wzwop(\Omega, {\mathcal{C}l_n})$.

The Euclidean Dirac operators $D_x$ and $D_0$ arise as generalizations of the Cauchy-Riemann operator of one complex variable. As homogenous linear differential operators,
$$D_x:=\sum_{i=1}^{n}e_i\partial_{x_i},$$
$$D_0:=e_0\partial_{x_0}+\sum_{i=1}^{n}e_i\partial_{x_i}=e_0\partial_{x_0}+D_x.$$
Note $D_x^2=-\Delta_x$, where $\Delta_x$ is the Laplacian in $\RN$, and $\Delta_{n+1}=D_0\overline{D_0}$, where $\overline{D_0}$ is the Clifford conjugate of $D_0$.

\begin{definition}
A $\ClN$-valued function $f(x)$ defined on a domain $\Omega$ in $\RN$ is called left monogenic if $$D_xf(x)=\sum_{i=1}^{n}e_i\partial_{x_i}f(x)=0.$$ Similarly, $f$ is called a right monogenic function if it satisfies $$f(x)D_x=\sum_{i=1}^{n}\partial_{x_i}f(x)e_i=0$$
\end{definition}

\par
Let $f \in C^1(\Omega, {\mathcal{C}l_n})$, $G(x-y)=\displaystyle\frac{\overline{x-y}}{\|x-y\|^{n+1}}$ beign the fundamental solution of $D_0$. Hence, the Cauchy transform is defined as
\begin{eqnarray*}
T_\Omega f(x)=\int_\Omega G(x-y)f(y)dy,
\end{eqnarray*}
where $T$ is the generalization of the Cauchy transform in the complex plane to Euclidean space, and it is the right inverse of $D_0$, that is $D_0T=I$. Also, the non-singular boundary integral operator is given by
\begin{eqnarray*}
F_{\partial \Omega}f(x)=\int_{\partial \Omega}G(x-y)n(y)f(y)d\sigma(y).
\end{eqnarray*}

We have the Borel-Pompeiu Theorem as follows.
\begin{theorem} (\cite{GK}) For $f\in C^1(\Omega,\ClN)\cap C(\bar\Omega)$, we have
\begin{eqnarray*}
f(x)=\int_{\partial \Omega}G(x-y)n(y)f(y)d\sigma(y)+\int_\Omega G(x-y)D_0f(y)dy,
\end{eqnarray*}
In particular, if $f\in \wzwo(\Omega,{\mathcal{C}l_n})$, then
\begin{eqnarray*}
f(x)=\int_\Omega G(x-y)D_0f(y)dy.
\end{eqnarray*}
\end{theorem}
\section{$\Pi$-operator in Euclidean space}
It is well known that in complex analysis, the $\Pi$-operator can be realized as the composition of $\partial_{\bar{z}}$ and the Cauchy transform. As the generalization to higher dimension in Clifford algebra, we have the $\Pi$-operator in $\RN$ defined as follows.
\begin{definition}
The $\Pi$-operator in Euclidean space $\RN$ is defined as
$$\Pi=\overline{D_0}T.$$
\end{definition}
The following are some well known properties for the $\Pi$-operator.
\begin{theorem} (\cite{GK})
Suppose $f\in \wzwop(\Omega)(1<p<\infty, k\geq 1)$, then
\begin{enumerate}
\item $D_0\Pi f=\overline{D_0}f,$
\item $\Pi D_0f=\overline{D_0}f-\overline{D_0}F_{\partial\Omega}f,$
\item $F_{\partial\Omega}\Pi f=(\Pi-T\overline{D_0})f,$
\item $D_0\Pi f-\Pi D_0f=\overline{D_0}F_{\partial\Omega}f.$
\end{enumerate}
\end{theorem}

The following decomposition of $L^2(\Omega,\ClN)$ helps us to observe that the $\Pi$-operator actually maps $L^2(\Omega,\ClN)$ to $L^2(\Omega,\ClN)$.
\begin{theorem} (\cite{GK}) \textbf{($L^2(\Omega,\ClN)$ Decomposition)}
$$L^2(\Omega,\ClN)=L^2(\Omega,\ClN)\bigcap Ker \overline{D_0}\bigoplus D_0(\wzwo(\Omega, \ClN)),$$
and
$$L^2(\Omega,\ClN)=L^2(\Omega,\ClN)\bigcap Ker D_0\bigoplus\overline{D_0}(\wzwo(\Omega, \ClN)).$$
\end{theorem}

Notice that, since
\begin{eqnarray*}
&&\Pi(L^2(\Omega,\ClN)\bigcap Ker \overline{D_0})=L^2(\Omega,\ClN)\bigcap Ker D_0,\\
&&\Pi(D_0(\wzwo(\Omega, \ClN)))=\overline{D_0}(\wzwo(\Omega, \ClN)),
\end{eqnarray*}
hence, this $\Pi$-operator is from $L^2(\Omega,\ClN)$ to $L^2(\Omega,\ClN)$.\\
\par
One key property of the $\Pi$-operator is that it is an $L^2$ isometry, in other words,

\begin{theorem} (\cite{Blaya})
For functions in $L^2(\Omega,\ClN)$, we have
\begin{eqnarray*}
\Pi^* \Pi=I.
\end{eqnarray*}
\end{theorem}
To complete this section, we give the classic example of the $\Pi$-operator solving the Beltrami equation. Let $\Omega \subseteq \RN$, $q:\Omega\rightarrow \ClN$ a bounded measurable function and $\omega:\Omega\rightarrow \mathcal{C}l_n$ be a sufficiently smooth function. The generalized Beltrami equation
$$D_0 \omega=q\overline{D_0}\omega$$
could be transformed into an integral equation
$$h=q(\overline{D_0}\phi+\Pi h)$$
where $\omega=\phi+Th$, which could have a unique solution if $\|q\|\leq q_0< \displaystyle\frac{1}{\|\Pi\|}$, see \cite{GK}, with $q_0$ being a constant. This tells us that the existence of a unique solution to the Beltrami equation depends on the norm estimate for the $\Pi$-operator.

\section{Construction and properties of spherical $\Pi$-type operator with generalized spherical Dirac operator}
\par
Recall that in one dimensional complex analysis, the $\Pi$-operator is defined as
$$\Pi f(z):=\partial_{\bar{z}}Tf(z)=\partial_{\bar{z}}\int_{\Omega}\displaystyle\frac{f(z)}{\eta-z}dz,$$
where $z=x+iy\in\mathbb{C}$ and $\partial_{\bar{z}}=\displaystyle\frac{\partial}{\partial x}+i\displaystyle\frac{\partial}{\partial y}$. This suggests us to generalize the $\Pi$-operator, we need to consider a variable $z$ with ``real" and ``imaginary" parts, so we can take conjugate of $\partial_z$ to define the $\Pi$-operator.

\subsection{Spherical $\Pi$-type operator with generalized spherical Dirac operator}

Let $\SN$ be the n-unit sphere. The spherical Dirac operator $D_{s,0}$ on $\SN$ is defined as follows.
$$\overline{x}D_0=\displaystyle\sum_{j=1}^{n} e_0e_j(x_0\partial_{x_j}-x_j\partial_{x_0})-\displaystyle\sum_{i=1,j>i}^{n} e_ie_j(x_i\partial_{x_j}-x_j\partial_{x_i})+\displaystyle\sum_{j=0}^n (x_j\partial_{x_j}).$$
Denote $\Gamma_0=\displaystyle\sum_{j=1}^{n} e_0e_j((x_0\partial_{x_j}-x_j\partial_{x_0}))-\displaystyle\sum_{i=1,j>i}^{n} e_ie_j((x_i\partial_{x_j}-x_j\partial_{x_i})).$
Hence,
$$D_{s,0}=\overline{x}^{-1}\overline{x}D_{s,0}=\frac{x}{\|x\|^2}(E_r+\Gamma_0)=\xi(D_r+\frac{\Gamma_0}{r}),$$
where $rD_r=E_r$, $r=\|x\|$ and $\xi \in \SN$.\\
In particular, we have the conformally invariant spherical Dirac operator as follows,
$$D_{s,0}=w(\Gamma_0-\frac{n}{2}).$$\\

Similarly, we have $\overline{D_{s}}=\overline{\xi}(D_r+\displaystyle \frac{\overline{\Gamma_0}}{r})$, and since $\overline{D_s}$ is also conformally invariant, we have $\overline{D_s}=\overline{w}(\overline{\Gamma_0}-\frac{n}{2})$, where
\begin{eqnarray*}
\overline{\Gamma_0}=-\displaystyle\sum_{j=1}^{n} e_0e_j(x_0\partial_{x_j}-x_j\partial_{x_0})-\displaystyle\sum_{i=1,j>i}^{n} e_ie_j(x_i\partial_{x_j}-x_j\partial_{x_i}).
\end{eqnarray*}
Here $\overline{D_s}$ is the Clifford involution of $D_s$.

\begin{lemma}
\begin{eqnarray*}
\Gamma_0\overline{w}=n\overline{w}-\overline{w}\overline{\Gamma_0};\\
\overline{\Gamma_0}w=nw-w\Gamma_0.
\end{eqnarray*}
\end{lemma}

\begin{proof}
The proof is similar to Theorem 3 in \cite{LR}.
\end{proof}

\begin{theorem} \label{Dsw}
$$D_s\overline{w}=-w\overline{D_s},\ \overline{D_s}w=-\overline{w}D_s.$$
\end{theorem}
\begin{proof}
Applying the last Lemma, a straight forward calculation completes the proof.
\end{proof}

\begin{theorem}

Since $D_s$ and $\overline{D_s}$ are both conformally invariant, we have their fundamental solutions as follows:
$$D_sG_s(w-v)=D_s\frac{\overline{w-v}}{\|w-v\|^n}=\delta(v),$$
$$\overline{D_s}\overline{G_s(w-v)}=\overline{D_s}\frac{w-v}{\|w-v\|^n}=\delta(v),$$
where $w,v\in \SN$.
\end{theorem}
\begin{proof}
The proof is similar to Proposition 4 in \cite{LR}.
\end{proof}

Let $\Omega$ be a bounded smooth domain in ${\SN}$ and $f \in C^1(\Omega, \ClN)$,  we have the Cauchy transforms for both $D_s$ and $\overline{D_s}$,
\begin{eqnarray*}
T_\Omega f(w)=\int_\Omega G_s(w-v)f(v)dv=\int_\Omega \frac{\overline{w-v}}{\|w-v\|^n}f(v)dv,\\
\overline{T}_\Omega f(w)=\int_\Omega \overline{G_s(w-v)}f(v)dv=\int_\Omega \frac{w-v}{\|w-v\|^n}f(v)dv.
\end{eqnarray*}

Also, the non-singular boundary integral operators are given by
\begin{eqnarray*}
F_{\partial \Omega}f(w)=\int_{\partial \Omega}G_s(w-v)n(v)f(v)d\sigma(v),\\
\overline{F}_{\partial \Omega}f(w)=\int_{\partial \Omega}\overline{G_s(w-v)}n(v)f(v)d\sigma(v)
\end{eqnarray*}
Then we have Borel-Pompeiu Theorem as follows.

\begin{theorem}(\cite{LR})\textbf{(Borel-Pompeiu Theorem)}
\par
For $f \in C^1(\Omega)\cap C(\bar\Omega)$, we have
\begin{eqnarray*}
f(w)=\int_{\partial \Omega}G_s(w-v)n(v)f(v)d\sigma(v)+\int_\Omega G_s(w-v)D_sf(v)dv, \label{1}
\end{eqnarray*}
in other words, $f=F_{\partial \Omega}f+T_\Omega D_sf$. Similarly, $f=\overline{F}_{\partial \Omega}f+\overline{T}_\Omega \overline{D_s}f$
\begin{eqnarray*}
f(w)=\int_{\partial \Omega}\overline{G_s(w-v)}n(v)f(v)d\sigma(v)+\int_\Omega \overline{G_s(w-v)}\overline{D_s}f(v)dv, \label{1}
\end{eqnarray*}
If $f$ is a function with compact support, then $TD_s=\overline{TD_s}=I$.
\end{theorem}

\par
Since the conformally invariant spherical Laplace operator $\Delta_s$ has the fundamental solution $H_s(w-v)=-\displaystyle\frac{1}{n-2}\frac{1}{\|w-v\|^{n-2}}$, see \cite{LR}. We have factorizations of $\Delta_s$ as follows.
\begin{theorem}
$\Delta_s=\overline{D_s}(D_s+w)=D_s(\overline{D_s}+\overline{w})$.
\end{theorem}
\begin{proof}
The proof is similar to Proposition 5 in \cite{LR}.
\end{proof}

We also have the dual of $D_s$ as follows.

\begin{theorem}
$D_s^*=-\overline{D_s}.$
\end{theorem}
\begin{proof}
Let $f,g:\Omega \rightarrow \Cln$ both have compact supports,
\begin{eqnarray*}
&&<D_sf,g>=<w(\Gamma_0-\frac{n}{2})f,g>=<(\Gamma_0-\frac{n}{2})f,\overline{w}g>\\
&=&<\Gamma_0 f,\overline{w}g>-\frac{n}{2}<f,\overline{w}g>=<f,\Gamma_0\overline{w}g>-\frac{n}{2}<f,\overline{w}g>\\
&=&<f,(n\overline{\omega}-\overline{\omega}\overline{\Gamma_0})g>-\frac{n}{2}<f,\overline{w}g>=<f, -\overline{w}(\overline{\Gamma_0}-\frac{n}{2})g>\\
&=&<f,-\overline{D_s}g>.
\end{eqnarray*}
\end{proof}

\begin{definition}
Define the generalized spherical $\Pi$-type operator as $$\Pi_{s,0}f=(\overline{D_s+w})Tf.$$
\end{definition}
We have some properties of $\Pi_{s,0}$ as follows.
\begin{proposition}
\begin{eqnarray*}
&&D_s\Pi_{s,0}=\overline{D_s-w},\\
&&\Pi_{s,0}D_s=\overline{D_s+w}.
\end{eqnarray*}
\end{proposition}
\begin{proof}
\begin{eqnarray*}
&&D_s\Pi_{s,0}=D_s(\overline{D_s+w})T=(\overline{D_s-w})D_sT=\overline{D_s-w},\\
&&\Pi_{s,0}D_s=(\overline{D_s+w})TD_s=\overline{D_s+w}.
\end{eqnarray*}
\end{proof}

From the proposition above, we can have decompositions of $L^2(\Omega,\Cln)$ as follows.
\begin{theorem}
$$L^2(\Omega,\Cln)=L^2(\Omega,\ClN)\bigcap Ker (\overline{D_s-w})\bigoplus D_s(\wzwo(\Omega, \ClN)),$$
$$L^2(\Omega,\Cln)=L^2(\Omega,\ClN)\bigcap Ker D_s\bigoplus(\overline{D_s+w})(\wzwo(\Omega, \ClN)).$$
\end{theorem}
Notice that
\begin{eqnarray*}
&&\Pi_{s,0}(L^2(\Omega,\ClN)\bigcap Ker  (\overline{D_s-w})=L^2(\Omega,\ClN)\bigcap Ker D_s,\\
&&\Pi_{s,0}D_s(\wzwo(\Omega, \ClN))=(\overline{D_s+w})(\wzwo(\Omega, \ClN)).
\end{eqnarray*}
Hence, $\Pi_{s,0}$ operator is from $L^2(\Omega,\ClN)$ to $L^2(\Omega,\ClN)$. The proof is similar to Theorem 1 in \cite{GK}.\\

\begin{definition}
We define the $\Pi_s^+$ operator as $$\Pi_s^+f=\overline{D_s}T^+f,$$
where $T^+f=\displaystyle\int_\Omega G^+(w-v)f(v)dv$,
$$G^+(w-v)=G_s(w-v)+wH_s(w-v)-2G_s^{(3)}(w-v),$$
and
$$G_s^{(3)}(w-v)=\displaystyle\frac{1}{(n-2)(n-4)}\frac{\overline{w-v}}{\|w-v\|^{n-4}}.$$
\end{definition}
Notice that $G_s^{(3)}(w-v)$ is actually the reproducing kernel of $D_s^{(3)}=(D_s-w)\overline{D_s}(D_s+w)$ and the proof is similar to a proof in \cite{LR}.

\begin{proposition}
\begin{eqnarray*}
\Pi_{s,0}(L^2(\Omega,\ClN)\bigcap Ker D_s)&=&L^2(\Omega,\ClN)\bigcap Ker  (\overline{D_s-w}),\\
\Pi_{s,0}(\overline{D_s+w})(\wzwo(\Omega, \ClN))&=&D_s(\wzwo(\Omega,\ClN)).
\end{eqnarray*}
\end{proposition}

\begin{theorem}
$\Pi_s$ is an isometry on $\wzwo(\Omega, \ClN)$ up to isomorphism. \label{isom up}
\end{theorem}
\begin{proof}
Let $f\in L^2(\Omega,\ClN)$, then
\begin{eqnarray*}
\langle\Pi_s f,\Pi_s^+ g\rangle&=&\langle (\overline{D_s}+\overline{w})Tf, \overline{D_s}T^+g \rangle\\
&=&\langle Tf, (-D_s+w)\overline{D_s}T^+ g \rangle\\
&=&-\langle Tf, (D_s-w)\overline{D_s} T^+ g \rangle\\
&=&-\langle Tf, \overline{D_s}(D_s+w)T^+ g \rangle\\
&=&\langle D_sTf,  (D_s+w)T^+ g \rangle=\langle f,g \rangle.
\end{eqnarray*}
\end{proof}

\subsection{Application of $\Pi_{s,0}$ to the solution of a Beltrami equation}
We have a Beltrami equation related to $\Pi_{s,0}$ as follows.
Let $\Omega \subseteq \Sn$ be a bounded, simply connected domain with sufficiently smooth boundary, $q: \Omega\longrightarrow \Cln$ a measurable function. Let $f: \Omega\longrightarrow \Cln$ be a sufficiently smooth function. The spherical Beltrami equation is as follows:
$$ D_sf=q(\overline{D_s+w})f.$$
It has a unique solution $f=\phi+Th$ where $\phi$ ia an arbitrary left-monogenic function such that $D_sf=0$ and $h$ is the solution of an  integral equation
$$h=q\big( (\overline{D_s+w})\phi+\Pi_{s,0} h\big).$$
By the Banach fixed point theorem, the previous integral equation has a unique solution in the case where
$$\|q\|\leq q_0< \frac{1}{\|\Pi_{s,0}\|},$$
with $q_0$ being a constant. Hence, for the rest of this section, we will estimate the $L^p$ norm of $\Pi_{s,0}$ with $p>1$.\\
\par
Since $\overline{D_s}=\overline{w}(\overline{\Gamma}-\frac{n}{2})=\overline{w}(w\overline{D_0}-E_r-\frac{n}{2})=\overline{D_0}-w E_r-\frac{n}{2}\overline{w}$, then
$$\Pi_{s,0} f(w)=\overline{(D_s+w)}Tf(w)=(\overline{D}T+\overline{w}(1-E_w)T-\frac{n}{2}T)f(w).$$
it is easy to see that
\begin{eqnarray*}
\frac{\partial}{\partial w_j}\int_{\SN}\frac{\overline{w-v}}{\|w-v\|^n}f(v)dv=
\int_{\SN} \frac{\overline{e_j}-n(w_j-v_j)\frac{\overline{w-v}}{\|w-v\|^2}}{\|w-v\|^n}f(v)dv+\omega_n\frac{\overline{e_j}}{n}f(v),
\end{eqnarray*}
since
\begin{eqnarray*}
\frac{\partial}{\partial w_j}\frac{\overline{w-v}}{\|w-v\|^n}=\frac{\overline{e_j}-n(w_j-v_j)\frac{\overline{w-v}}{\|w-v\|^2}}{\|w-v\|^n},
\end{eqnarray*}
and using Chapter IX $\S$ 7 in \cite{S}
\begin{eqnarray*}
\int_{S}\frac{\overline{w-v}}{\|w-v\|} \cos(r,w_j)dS=\omega_n\frac{\overline{e_j}}{n},
\end{eqnarray*}
where $S$ is a sufficiently small neighborhood of $w$ on $\SN$.\\
Hence, we have
\begin{eqnarray*}
\overline{D}Tf(w)&=&
\frac{1}{\omega_n}\int_{\SN} \frac{\sum \overline{e_j}^2-n\sum(w_j-v_j)\overline{e_j}\frac{\overline{w-v}}{\|w-v\|^2}}{\|w-v\|^n}f(v)dv+ \frac{\sum \overline{e_j}^2}{n}f(v)\\
&=&\frac{1}{\omega_n}\int_{\SN} \frac{(1-n)-n\frac{\overline{w-v}^2}{\|w-v\|^2}}{\|w-v\|^n}f(v)dv+\frac{1-n}{n}f(v)\\
E_wTf(w)&=&\frac{1}{\omega_n}\int_{\SN} \frac{\sum w_j\overline{e_j}-n\sum w_j(w_j-v_j)\frac{\overline{w-v}}{\|w-v\|^2}}{\|w-v\|^n}f(v)dv+ \frac{\sum w_j\overline{e_j}}{n}f(v)\\
&=&\frac{1}{\omega_n}\int_{\SN} \frac{\overline{w}-n<w,w-v>\frac{\overline{w-v}}{\|w-v\|^2}}{\|w-v\|^n}f(v)dv+\frac{\overline{w}}{n}f(v).
\end{eqnarray*}
Therefore, we have an integral expression of $\Pi_{s,0}$ as follows.
\begin{theorem}
\begin{eqnarray*}
\Pi_{s,0} f(w)&=&(\overline{D}T+\overline{w}(1-E_w)T-\frac{n}{2}T)f(w)\\
&=&\frac{1}{\omega_n}\int_{\SN}\frac{1-n-\overline{w}^2}{\|w-v\|^n}f(v)dv+\frac{n}{\omega_n}\int_{\SN} \frac{\overline{v}-\langle w,v\rangle\overline{w}}{\|w-v\|^{n+1}}\cdot \frac{\overline{w-v}}{\|w-v\|}f(v)dv\\
&+&(1-\frac{n}{2})\frac{\overline{w}}{\omega_n}\int_{\SN} \frac{\overline{w-v}}{\|w-v\|^n}f(v)dv+\frac{1-n}{n}f(v).
\end{eqnarray*}
\end{theorem}
Since
$$\Pi_{s,0}=(\overline{D_s+w})T=(\overline{w}(\overline{\Gamma_0}-\frac{n}{2})+\overline{w})T=\overline{w}\overline{\Gamma_0}T+(1-\frac{n}{2})\overline{w}T,$$
where $\overline{\Gamma_0}=-\displaystyle\sum_{j=1}^{n} e_0e_j(x_0\partial_{x_j}-x_j\partial_{x_0})-\displaystyle\sum_{i=1,j>i}^{n} e_ie_j(x_i\partial_{x_j}-x_j\partial_{x_i})$. To estimate the $L^p$ norm of $\Pi_{s,0}$, we need the following result. \\
\begin{theorem}\label{Lp}
Suppose $p$ is a positive integer and $p>1$, then
$\|T\|_{L^p}\leq\displaystyle \frac{\omega_{n-1}}{4}.$
\end{theorem}
\begin{proof}
Since
\begin{eqnarray*}
\|Tf\|_{L^p}^p&=&(\frac{1}{\omega_n})^p\int_\Omega \|\int_\Omega G_s(w-v)f(v)dv^n\|^p dw^n\\
&=&(\frac{1}{\omega_n})^p\int_\Omega \|\int_\Omega G_s(w-v)^{\frac{1}{q}}G_s(w-v)^{\frac{1}{p}}f(v)dv^n\|^p dw^n\\
&\leq& (\frac{1}{\omega_n})^p\int_\Omega \big((\int_\Omega \|G_s(w-v)\|dv^n)^{\frac{p}{q}} \cdot \int_\Omega \|G_s(w-v)\| \|f(v)\|^p dv^n \big)dw^n\\
&\leq& (\frac{1}{\omega_n})^pC_1^{\frac{p}{q}} \int_\Omega \int_\Omega \|G_s(w-v)\| \|f(v)\|^p dv^n dw^n\\
&=& (\frac{1}{\omega_n})^pC_1^{\frac{p}{q}} \int_\Omega \|f(v)\|^p (\int_\Omega \|G_s(w-v)\|dw^n) dv^n\\
&\leq& (\frac{1}{\omega_n})^pC_1^{\frac{p}{q}+1} \int_\Omega \|f(v)\|^p (\int_\Omega \|G_s(w-v)\|dw^n) dv^n\\
&=&(\frac{1}{\omega_n})^pC_1^p \cdot \int_\Omega \|f(v)\|^p dv^n\\
&=&(\frac{1}{\omega_n})^pC_1^p \cdot \|f\|_{L^p}^p
\end{eqnarray*}
where $p,\ q> 1$ are positive integers and $\displaystyle\frac{1}{p}+\displaystyle\frac{1}{q}=1$, where
\begin{eqnarray*}
C_1\leq\big| \int_{\SN}\|G_s(w-v)\|dv^n \big|=\big| \int_{\SN}\frac{1}{\|w-v\|^{n-1}}dv^n \big|.
\end{eqnarray*}
Due to the symmetry we can choose any fixed point $w$, hence we choose $w=(1,0,0,...,0)$ and $v=(x_0, x_1,\cdots,x_n)\in \SN$, i.e. $\displaystyle\sum_{i=0}^n\|x_i\|^2=1$.
Let $v=\cos\theta e_0+ \sin\theta \zeta$, where $\zeta$ is a vector on $n-1$-sphere, then we have $dv^n=\sin^{n-1}\theta d\theta$,
\begin{eqnarray*}
&&\int_{\SN} \frac{1}{[2(1-x_1)]^{\frac{n-1}{2}}}dv^n\\
&=&2^{-\frac{n-1}{2}} \int_0^\pi \frac{1}{(1-\cos\theta)^{\frac{n-1}{2}}}\sin^{n-1}\theta d\theta\\
&=&2^{-\frac{n-1}{2}} \int_0^\pi (2\sin^2\frac{\theta}{2})^{-\frac{n-1}{2}}(2\sin\frac{\theta}{2}\cos\frac{\theta}{2})^{n-1}d\theta\\
&=& \int_0^\pi \cos^{n-1}\frac{\theta}{2} d\theta\\
&=& 2\cdot \frac{1}{2}\cdot \displaystyle\frac{\Gamma(\frac{1}{2})\Gamma(\frac{n}{2})}{\Gamma(\frac{n-1}{2}+1)}\\
&=& \sqrt{\pi}\displaystyle\frac{\Gamma(\frac{n}{2})}{\Gamma(\frac{n+1}{2})}.
\end{eqnarray*}
Since $\omega_n=\displaystyle\frac{2\pi^{(n+1)/2}}{\Gamma(\frac{n+1}{2})}$, we have
$\|T\|_{L^p}\leq \displaystyle\frac{\omega_{n-1}}{4}$.
\end{proof}
\par
Let $G_0$ be the operator defined by
\begin{eqnarray*}
G_0 g(w)=-\frac{1}{(n-1)\omega_n}\int_{\SN}\frac{1}{\|w-v\|^{n-1}}g(v)dv,\ n\geq 3,
\end{eqnarray*}
and $R_s=\overline{\Gamma_0}\circ G_0$ is a Riesz transformation of gradient type (see \cite{AL}).
Then we have,

\begin{proposition}\cite {AL},
The operator $R_s$ is a $L^p$ operator and the $L^p$ norm is bounded by
$$\frac{\pi^{1/2}}{2\sqrt2}(\frac{p}{p-1})^{1/2}B_p,$$
where $B_p=C_{M,p}+C_p$, $C_{M,p}$ is the $L^p$ norm of the maximal truncated Hilbert transformation on $\mathbb S^1$, and $C_p=cot\frac{\pi}{2p^{*}},\ \frac{1}{p}+\frac{1}{p^*}=1$.
\end{proposition}
Hence,
\begin{eqnarray}
\nonumber
&&\|\overline{\Gamma_0}\frac{1}{\omega_n}\int_\Omega \frac{1}{\|w-v\|^{n-1}}\cdot \frac{\overline{w-v}}{\|w-v\|}f(v)dv\|_{L^p}\\
\nonumber &\leq& (n-1)\frac{\pi^{1/2}}{2\sqrt2}(\frac{p}{p-1})^{1/2}B_p \|f(v)\|_{L^p}\\
&=& (n-1)\frac{\pi^{1/2}}{2\sqrt2}(\frac{p}{p-1})^{1/2}B_p \|f(v)\|_{L^p} \label{eqn1}.
\end{eqnarray}

Recall that $\Pi_{s,0}f=(\overline{D_s}+\overline{w})Tf=(\overline{w}(\overline{\Gamma_0}-\frac{n}{2})+\overline{w})Tf=\overline{w}\overline{\Gamma_0}Tf+(1-\frac{n}{2})\overline{w}Tf$,

\par
and by Theorem \ref{Lp},
\begin{eqnarray}
||(1-\frac{n}{2})\overline{w}Tf||_{L^p}=\|(1-\frac{n}{2})\frac{\overline{w}}{\omega_n}\int_\Omega \frac{\overline{w-v}}{\|w-v\|^n}f(v)dv\|_{L^p}
\leq (\frac{n}{2}-1) \displaystyle\frac{\omega_{n-1}}{4}\|f\|_{L^p}\label{eqn2}.
\end{eqnarray}

By inequalities (\ref{eqn1}) and (\ref{eqn2}), we show that $\Pi_{s,0}$ is a bounded operator mapping from $L^p$ space to itself, and
\begin{eqnarray*}
\|\Pi_{s,0} \|_{L^p}\leq (n-1)\frac{\pi^{1/2}}{2\sqrt2}(\frac{p}{p-1})^{1/2}B_p+(\frac{n}{2}-1) \displaystyle\frac{\omega_{n-1}}{4}.
\end{eqnarray*}

\textbf{Remark:} The spherical $\Pi$-type operator $\Pi_{s,0}$ preserves most properties of the $\Pi$ operator in Euclidean space and more importantly, it is a singular integral operator which helps to solve the corresponding Beltrami equation. Unfortunately, it is also only an $L^2$ isometry up to isomorphism as shown in Theorem \ref{isom up} . In the next section, we will use the spectrum theory of differential operators to claim that there is a spherical $\Pi$-type operator which is also an $L^2$ isometry.

\section{Eigenvectors of spherical Dirac type operators}
In this section, we will investigate the spectrums of several spherical Dirac type operators and the spherical Laplacian. During the investigation, we will point out there is a spherical $\Pi$-type operator which is an $L^2$ isometry.\\
\par
Since $\Gamma_0=\overline{x}D_0-E_r$, it is easy to verify the fact that if $p_m$ is a monogenic polynomial and is homogeneous with degree $m$, that is $D_0f_m=0$ and $E_rf_m=mf_m$, then $\Gamma_0f_m=-mf_m$, so $f_m$ is an eigenvector of $\Gamma_0$ with eigenvalue $-m$. Similarly, if $\overline{D_0}g_m=0$, $g_m$ is an eigenvector of $\overline{\Gamma_0}$ with eigenvalue $-m$.\\
\par
Let $\mathcal{H}_k$ be the space of $\ClN$-valued harmonic polynomials homogeneous of degree k and $\Mk$ be the $\ClN$-valued monogenic polynomials homogeneous of degree k, $\overline{\Mk}$ is the clifford involution of $\Mk$. By an Almansi-Fischer decomposition \cite{DLRV}, $\mathcal{H}_k=\Mk \bigoplus \bar{x}\overline{\Mkk}$. Hence, for for all harmonic functions with homogeneity of degree $k$, there exist $p_k\in Ker D_0$, and $p_{k-1}\in Ker \overline{D_0}$ such that $h_k=p_k+\bar{x}\overline{p_{k-1}}$. Then, it is easy to get that $\Gamma_0p_k=-kp_k$ and $\Gamma_0\bar{x}\overline{p_{k-1}}=(n+k)\bar{x}\overline{p_{k-1}}$. \\
\par
Let $H_m$ denote the restriction to $\SN$ of the space of $\ClN$-valued harmonic polynomials with homogeneity of degree $m$. $P_m$ is the space of spherical $\ClN$-valued left monogenic polynomials with homogeneity of degree $-m$ and $Q_m$ is the space of spherical $\ClN$-valued left monogenic polynomials with homogeneity of degree $n+m$, $m=0,1,2,...$.Then we have $H_m=P_m\bigoplus Q_m$ (\cite{Balinsky}). It is well known that $L^2(\SN)=\displaystyle\sum_{m=0}^\infty H_m$ (\cite{Ax}), it follows $L^2(\SN)=\displaystyle\sum_{m=0}^\infty P_m\bigoplus Q_m$.
If $p_m\in P_m$, since $\Gamma_0 p_m=-mp_m$,  it is an eigenvector of $\Gamma_0$ with eigenvalue $-m$, and for $q_m\in Q_m$, it is an eigenvector of $\Gamma_0$ with eigenvalue $n+m$. Therefore, the spectrum of $\Gamma_0$ is $\sigma(\Gamma_0)=\{-m, m=1,2,...\}\cup \{m+n, m=0,1,2,...\},$. Since $D_s=w(\Gamma_0-\displaystyle\frac{n}{2})$, the spectrum of $D_s$ is $\sigma(D_s)=\sigma(\Gamma_0)-\frac{n}{2}$, which is $\{-m-\frac{n}{2}, m=0,1,2,...\}\cup \{m+\frac{n}{2}, m=0,1,2,...\}$.\\
\par
As mentioned in the previous section that $D_sT=TD_s=I$, and we know that $D_s: P_m\longrightarrow Q_m$ (\cite{Balinsky}). Hence,we have $T:\ Q_m\longrightarrow P_m$ and the spectrum of $T$ is the reciprocal of the spectrum of $D_s$, which is $\sigma(T)=\{\frac{1}{m+\frac{n}{2}}, m=0,1,2,...\}\bigcup\{\frac{1}{-m-\frac{n}{2}}, m=0,1,2,...\}$. Similar arguments apply for $\overline{D_s}$ and  $\overline{T}$, in fact $\sigma(\overline{D_s})=\sigma(D_s)$£¬ and  $\sigma(\overline{T})= \sigma(T)$.\\
\par
Now with similar strategy as in \cite{Balinsky}, we consider the operator $\overline{D_s}T$ which maps $L^2(\SN)$ to $L^2(\SN)$. If $u\in C^1(\SN)$ then $u\in L^2(\SN)$. It follows that
\begin{eqnarray*}
u=\displaystyle\sum_{m=0}^\infty\sum_{p_m\in P_m}p_m+\sum_{m=0}^{-\infty}\sum_{q_m\in Q_m}q_m,
\end{eqnarray*}
where $p_m$ and $q_m$ are eigenvectors of $\Gamma_0$. Further the eigenvectors $p_m$ and $q_m$ can be chosen so that within $P_m$ they are mutually orthogonal. The same can be done for the eigenvectors $q_m$. Moreover, as $u\in C^1(\SN)$ then $\overline{D_s}Tu\in C^0(\SN)$ and so $\overline{D_s}Tu\in L^2(\SN)$. Consequently,

\begin{eqnarray*}
&&\overline{D_s}Tu=\displaystyle\sum_{m=0}^\infty\sum_{p_m\in P_m}\overline{D_s}Tp_m+\sum_{m=0}^{\infty}\sum_{q_m\in Q_m}\overline{D_s}Tq_m\\
&=&\displaystyle\sum_{m=0}^\infty\sum_{q_m\in Q_m}\overline{D_s}\frac{1}{m+\frac{n}{2}}q_m+\sum_{m=0}^{\infty}\sum_{p_m\in P_m}\overline{D_s}\frac{1}{-m-\frac{n}{2}}p_m
\end{eqnarray*}
and
\begin{eqnarray*}
&&||\overline{D_s}Tu||^2_{L^2}=
\displaystyle\sum_{m=0}^\infty(\frac{1}{m+\frac{n}{2}})^2\sum_{q_m\in Q_m}\|\overline{D_s}q_m\|_{L^2}+\sum_{m=0}^{\infty}(\frac{1}{-m-\frac{n}{2}})^2\sum_{p_m\in P_m}\|\overline{D_s}p_m||_{L^2}\\
&=&\displaystyle\sum_{m=0}^\infty(\frac{1}{m+\frac{n}{2}})^2(m+\frac{n}{2})^2\sum_{p_m\in P_m}\|p_m||_{L^2}+\sum_{m=0}^{\infty}(\frac{1}{-m-\frac{n}{2}})^(-m-\frac{n}{2})^2\sum_{q_m\in Q_m}\|q_m||_{L^2}\\
&=&\displaystyle\sum_{m=0}^\infty\sum_{p_m\in P_m}||p_m||_{L^2}+\sum_{m=0}^{\infty}\sum_{q_m\in Q_m}||q_m||_{L^2}\\
&=&||u||_{L^2}.
\end{eqnarray*}
This shows $\overline{D_s}T$ is an $L^2(\SN)$ isometry.\\
\par
By the help of the spectrum of $T$, we have the $L^2$ norm estimate of the $\Pi_{s,0}$, that is
\begin{eqnarray*}
\|\Pi_{s,0}u\|_{L^2}&\leq &\|\overline{D_s}Tu\|_{L^2}+\|\overline{w}\|_{L^2}\|Tu\|_{L^2}\\
&=&\|u\|_{L^2}+(\frac{1}{m+\frac{n}{2}})^2(\sum_{m=0}^{\infty}\sum_{p_m\in P_m}\|p_m\|_{L^2}+\sum_{m=0}^{\infty}\sum_{q_m\in Q_m}\|q_m\|_{L^2})\\
&\leq & (1+\frac{4}{n^2})\|u\|_{L^2}.
\end{eqnarray*}
Hence we have $\|\Pi_{s,0}\|_{L^2}\leq 1+\displaystyle\frac{4}{n^2}$.\\

\par
By Theorem 13, $\Delta_s=\overline{D_s}(D_s+w)=(\overline{D_s}-\overline{w})D_s=D_s(\overline{D_s}+\overline{w})=(D_s-w)\overline{D_s}$.\\
Since $D_s=w(\Gamma_0-\displaystyle\frac{n}{2}),\ \overline{D_s}=\overline{w}(\overline{\Gamma_0}-\displaystyle\frac{n}{2})$, a straightforward calculation shows us that
\begin{eqnarray*}
\Delta_s&=&-(\Gamma_0-\frac{n}{2})^2-\overline{w}w(\Gamma_0-\frac{n}{2})=-\Gamma_0^2+(n-1)\Gamma_0-(\frac{n^2}{4}-\frac{n}{2})\\
&=&-(\overline{\Gamma_0}-\frac{n}{2})^2-\overline{w}w(\overline{\Gamma_0}-\frac{n}{2})=-\overline{\Gamma_0}^2+(n-1)\overline{\Gamma_0}-(\frac{n^2}{4}-\frac{n}{2}).
\end{eqnarray*}
Since for $0<r<1$, any harmonic function $h_m\in B(0,r)=\{x\in\Rn:\ ||x||<r\}$ with homogeneity degree m, we have $h_m=f_m+g_m$, where $f_m\in Ker D_0$ and $g_m\in\overline{D_0}$, they are both homogeneous with degree m (see Lemma 3 \cite{FS}). Consequently,
\begin{eqnarray*}
\Delta_sf_m=(-\Gamma_0^2+(n-1)\Gamma_0-(\frac{n^2}{4}-\frac{n}{2}))f_m=(-m^2-m(n-1)-(\frac{n^2}{4}-\frac{n}{2}))f_m,
\end{eqnarray*}
and
\begin{eqnarray*}
\Delta_sg_m=(-\overline{\Gamma_0}^2+(n-1)\overline{\Gamma_0}-(\frac{n^2}{4}-\frac{n}{2}))g_m=(-m^2-m(n-1)-(\frac{n^2}{4}-\frac{n}{2}))g_m.
\end{eqnarray*}
Hence
\begin{eqnarray*}
\Delta_sh_m&=&\Delta_s(f_m+g_m)=(-m^2-m(n-1)-(\frac{n^2}{4}-\frac{n}{2}))(f_m+g_m)\\
&=&(-m^2-m(n-1)-(\frac{n^2}{4}-\frac{n}{2}))h_m.
\end{eqnarray*}
Since for any function $u\in L^2(\SN): \Omega \mapsto \ClN$, $u=\displaystyle\sum_{m=0}^\infty h_m$, where $h_m\in H_m$, it follows that $\Delta_s$ has spectrum $\sigma(\Delta_s)=\{-m^2-m(n-1)-(\frac{n^2}{4}-\frac{n}{2}):m=0,1,2,...\}$.

\par
In order to preserve the property of isometry of the $\Pi$-operator on the sphere, we define the isometric spherical $\Pi$-operator as $\Pi_{s,1}$ as $\Pi_{s,1}=\overline{D_s}T$, which is isometry in $L^2$ space. We can solve the Beltrami equation related to $\Pi_{s,1}$ as follows.\\
\par
Let $\Omega \subseteq \Sn$ be a bounded, simply connected domain with sufficiently smooth boundary, and $q, f: \Omega\longrightarrow \Cln$,  q is a measurable function, and f is sufficiently smooth. The spherical Beltrami equation is as follows:
$$ D_sf=q\overline{D_s}f.$$
It has a unique solution $f=\phi+Th$ where $\phi$ ia an arbitrary left-monogenic function such that $D_sf=0$ and $h$ is the solution of an  integral equation
$$h=q( \overline{D_s}\phi+\Pi_{s,1} h).$$
By the Banach fixed point theorem, the previous integral equation has a unique solution in the case of
$$\|q\|\leq q_0< \frac{1}{\|\Pi_{s,1}\|}£¬$$
with $q_0$ being a constant. Hence, we can use the estimate of the $L^p$ norm of $\Pi_{s,1}$ with $p>1$, where
\begin{eqnarray*}
\|\Pi_{s,1} \|_{L^p}\leq (n-1)\frac{\pi^{1/2}}{2\sqrt2}(\frac{p}{p-1})^{1/2}B_p+\frac{n}{2} \displaystyle\frac{\omega_{n-1}}{4}.
\end{eqnarray*}

Wanqing Cheng\\
Department of Mathematical Science, University of Arkansas, Fayetteville, Arkansas, 72701, USA. Email: wcheng@uark.edu\\
\par
John Ryan\\
Department of Mathematical Science, University of Arkansas, Fayetteville, Arkansas, 72701, USA. Email: jryan@uark.edu\\
\par
Uwe K\"{a}hler\\
 Mathematics Department, Universidade de Aveiro, 3810-193 Aveiro, Portugal. Email: ukaehler@ua.pt

\end{document}